\documentclass[a4paper,english,oneside]{amsart}
\usepackage[utf8]{inputenc}
\usepackage[T1]{fontenc}
\usepackage[hscale=0.68, vscale=0.7]{geometry}
\usepackage[numbers]{natbib}

\usepackage{amssymb}
\usepackage{mathtools}

\DeclareFontFamily{OMX}{mlmex}{}
\DeclareFontShape{OMX}{mlmex}{m}{n}{%
   <->mlmex10%
   }{}%
\usepackage{mlmodern}

\theoremstyle{plain}
\newtheorem{theorem}{Theorem}
\newtheorem{proposition}[theorem]{Proposition}

\theoremstyle{definition}

\newtheorem{remark}[theorem]{Remark}

\usepackage{color}

\usepackage[np,autolanguage]{numprint}
\AtBeginDocument{\npthousandthpartsep{\,}}

\usepackage{hyperref}
\hypersetup{%
  colorlinks=true,%
  citecolor=[RGB]{120,29,126},
  pdfauthor={Jean-François Burnol},%
  pdfsubject={Digamma function, Kempner series},%
  pdfstartview=FitH,%
  pdfpagemode=UseNone,%
}

\newcommand\plusbas{\vphantom{X^X}}

\newcommand\cX[1][]{%
  \if\relax\detokenize{#1}\relax\else{}_{\plusbas#1}\mskip-1mu\relax\fi
  \mathcal{X}%
}
\newcommand\cY[1][]{%
  \if\relax\detokenize{#1}\relax\else{}_{\plusbas#1}\mskip-1mu\relax\fi
  \mathcal{Y}%
}
\newcommand\cA{\mathcal{A}}

\newcommand\sD[1][]{%
  \if\relax\detokenize{#1}\relax\else{}_{\plusbas#1}\mskip-1mu\relax\fi
  \mathsf{D}%
}
\newcommand\bsD[1][]{%
  \if\relax\detokenize{#1}\relax\else{}_{\plusbas#1}\mskip-1mu\relax\fi
  \boldsymbol{\mathsf{D}}%
}

\newcommand\dx{\mathrm{d}\mskip-1mu x}

\allowdisplaybreaks

\usepackage{setspace}

\title[Digamma and Kempner]{%
  Digamma function and general Fischer series in the theory of Kempner sums}

\author[J.-F. Burnol]{Jean-François Burnol}

\address{Université de Lille,
  Faculté des Sciences et technologies,
  Département de mathématiques,
  Cité Scientifique,
  F-59655 Villeneuve d'Ascq cedex,
  France}
\email{jean-francois.burnol@univ-lille.fr}
\date{v1-2: March 6-7, 2024; v3: April 27, 2024}


\usepackage{eso-pic}
\usepackage{xcolor}
\AddToShipoutPictureFG*{%
\AtTextUpperLeft{\raisebox{1cm}{{\upshape\textcolor{gray}{\texttt{VERSION OF
      27-04-2024 19:17:15 CEST}}}}}%
}

\begin{document}

\begin{abstract}
  The harmonic sum of the integers which are missing $p$ given digits in a
  base $b$ is expressed as $b \log(b)/p$ plus corrections indexed by the
  excluded digits and expressed as integrals involving the digamma function
  and a suitable measure.  A number of consequences are derived, such as
  explicit bounds, monotony, series representations and asymptotic expansions
  involving the zeta values at integers, and suitable moments of the measure.
  In the classic Kempner case of $b=10$ and $9$ as the only excluded digit,
  the series representation turns out to be exactly identical with a result
  obtained by Fischer already in 1993.  Extending this work is indeed the goal
  of the present contribution.
\end{abstract}

\maketitle

\onehalfspacing

\section{Introduction}

Let $b$ be an integer at least equal to $2$ and
$A\subset\sD=\{0,\dots,b-1\}$ a set of digits, and $E=\sD\setminus A$.  It is
assumed throughout that $E$ is not empty.

We call $A$ the set of admissible digits and $E$ the set of excluded digits.
We let $K(b,E)$ be the infinite sum of the reciprocals $1/m$ of positive
integers $m$ having no digits in base $b$ from the excluded set $E\subset
\sD$.  The cases with the admissible set $A$ empty or equal to $\{0\}$ are a
bit special as then $E$ contains all positive digits, so $K(b,E)$ is an empty
sum and evaluates to zero.  But the methods and formulas work, too, as we
shall see later.

It is at the level of introductory courses in the theory of series that such
infinite (or empty...) subsums of the harmonic series converge.  We refer to
\cite{kempner} and \cite{irwin} for the original studies, by now more than one
hundred years old, by Kempner and Irwin (see also
\cite[Thm. 144]{hardywright}).  With $N=\#A$ the cardinality of allowed
digits, one needs (if $N>1$) roughly to keep $N$ times as many terms as before
to reduce the distance to the limit by only a factor $b/N$.  This makes for
very slow convergence and the matter of numerical computation has attracted
widespread attention (see the works of Baillie \cite{baillie1979,baillie2008}
and Schmelzer-Baillie \cite{schmelzerbaillie} and comments therein).

The contents of this paper are organized as follows: in the first section we
quickly review measure-theoretic tools we have introduced in our previous
works \cite{burnolkempner} and \cite{burnolirwin}, where we obtained, starting
from some ``log-like'' formula
\begin{equation}\label{eq:K}
  K(b,E) = \int_{[b^{-1},1)}\frac{\mu(\dx)}{x}\;,
\end{equation}
multiple explicit representations of $K(b,E)$ as series with geometric
convergence.  These series coefficients involve the moments $u_m=\mu(x^m)$ or
complementary moments $v_m=\mu((1-x)^m)$ which are obtainable by linear
recurrences, hence the exact theoretical formulas convert into efficient
numerical algorithms (or rather, efficient enough for obtaining dozens, or hundreds or
thousands of decimal digits in reasonable time for reasonable tasks...).  The
focus for deeper understanding is on the properties of $u_m$ and $v_m$ as
analytic (rational) functions of the base $b$ (and excluded digits from $E$),
but we shall use only little of that in the present article.  The measure
$\mu$ on $[0,1)$, which sometimes will be denoted $\mu_{E}$ or even
$\mu_{b,E}$ to stress its dependency on $b$ and $E$, is an enumerable
combination of weighted Dirac masses chosen for \eqref{eq:K} to hold.

The second section obtains our main result Theorem \ref{thm:main}: an
integral formula for $K(b,E) - \frac{b}{\#E} \log(b)$ in terms of the digamma
function $\psi(x) = \frac{\mathrm{d}}{\dx}\log\Gamma(x)$ and the measure
$\mu$.  We explore first some immediate consequences regarding bounds for
$K(b,E)$.  For example we shall prove the strict upper bound $K(b,E) <
\frac {b}{\#E} \log(b)$ with the sole exception of
$E=\{0\}$, for which $b\log(b) < K(b,\{0\}) < b\log(b) + \zeta(2)b^{-1}$.  We will
also prove that among sets $E$ of cardinality $p<b$ the unique one providing
the maximal value of $K(b,E)$ is $\{0,b-1, \dots, b-p+1\}$.  Most of
this proof will be via an elementary combinatorial argument using only the
original definition of $K(b,E)$ but we will need the information of Theorem
\ref{thm:main} to handle the last step which is to prove that
$K(b,\{b-1,\dots,b-p\})$ is less than $K(b,\{0,b-1, \dots, b-p+1\})$ (except
of course when $p=b$).

In the third section we use the Taylor series of the digamma function at $1$
to transform the integral formula of Theorem \ref{thm:main} into a series (see
Theorem \ref{thm:series}) whose coefficients are given in terms of the
$\zeta(n)$'s, $n\geq2$, and suitable combinations of moments of the measure
$\mu_{b,E}$.  We then focus on the two cases where that formula is the most
efficient for theoretical studies: $E=\{0\}$ and $E=\{b-1\}$, and prove for
those two the existence of an asymptotic expansion of $K(b,E)- b\log(b)$ in
descending powers of $b$ to all orders.  These expansions start as:
\begin{align}
  \label{eq:2}
  K(b,\{0\}) &= b\log(b)+\frac{\zeta(2)}{2b}-\frac{\zeta(3)}{3b^2} -
  \frac{2\zeta(2)-\zeta(4)}{4b^3} + O(b^{-4})
\\
  \label{eq:3}
  K(b,\{b-1\}) &= b\log(b) - \frac{\zeta(2)}{2b} -
  \frac{3\zeta(2)+\zeta(3)}{3b^2} - \frac{2\zeta(2)+4\zeta(3)+\zeta(4)}{4b^3}
  + O(b^{-4})
\end{align}
By a completely different method we had established \eqref{eq:3} earlier in
\cite{burnollargeb}.  Here we justify the explicit values in \eqref{eq:2} and
\eqref{eq:3} only for the first two terms.  The $b^{-3}$ term requires
continuing our efforts but this would take up too much room and would make
sense only in the framework of a more advanced study of the measure $\mu$.
This is what we have done in a sequel \cite{burnolasymptotic}, where the
reader will find explicit formulas for the first five terms, and also the
analogous, but more complex, case with $E=\{d\}$ for a fixed excluded digit
$d>0$:
\begin{equation}\label{eq:1}
  K(b,\{d\}) = b\log(b) - b\log(1+\frac1d) 
             + \frac{(d+\frac12)(\zeta(2)-d^{-2})}{b} + O(b^{-2})
\end{equation}
In \cite{burnolasymptotic} we give explicitly all terms up to $b^{-5}$
inclusive, but they are too voluminous to reproduce here.  In the present
article we do not even prove $K(b,\{d\}) = b\log(b)-b\log(1+d^{-1})+O(b^{-1})$
because it (seemingly) needs additional ideas which would lengthen our
account and which are explained in \cite{burnolasymptotic}.

In the fourth section we explain how to set up linear recurrences for the
coefficients multiplying zeta values in the series from Theorem
\ref{thm:series}.  It turns out that besides the linear recurrences in the
style of \cite{burnolkempner,burnolirwin}, there is another set of linear
recurrences for the same quantities.

In the brief final section we consider the original Kempner series
$K(10,\{9\})$ which had been studied by Fischer \cite{fischer} via a linear
functional on continuous functions, and where the digamma function also was
used, leading to a series representation of $K(10,\{9\})$ involving
$10\log(10$, the zeta values $\zeta(n)$, and certain coefficients $\beta_n$.
It turns out that the Fischer $\beta_n$'s are but the complementary moments
$\mu((1-x)^n)$ of the measure $\mu$, in this speficic case of $b=10$ and
$E=\{9\}$, so the series of Theorem \ref{thm:series} in that case is exactly
the Fischer series.  This is not so much a surprise because of course our main
motivation which triggered the present research was indeed to understand the
Fischer result.  Other series can be derived from our main Theorem
\ref{thm:main} and extra steps are needed to reach results such as the
asymptotic expansion \eqref{eq:1}.  We invite the reader wishing to pursue the
matter further to read our follow-up extensive investigation
\cite{burnolasymptotic}.

\section{Measure and moments}

We use the vocabulary and constructions of \cite{burnolkempner,burnolirwin}.
The core tools there are a measure $\mu_{b,E}$ on $[0,1)$, which we will
simply write here as $\mu$ and whose precise definition is given next, and its
associated Stieltjes (up to replacement $n\to-n$) transforms:
\begin{equation}
  \label{eq:Un}
    U(n) = \int_{[0,1)}\frac{\mu(\dx)}{n+x}
\end{equation}
The measure $\mu$ is constructed in such a manner that $U(n)$ for positive
integers $n$ admits the following alternative expression:
\begin{equation}\label{eq:Unseries}
  U(n) = \sum_{\substack{m\text{ has }n\text{ as leading part}\\
                         \text{and all added digits are admissible}}} \frac1m
\end{equation}
To achieve this each ``admissible string'' $X$, i.e.\@ element of the subset
$\cX = \cup_{l\geq0}A^l$ of $\bsD = \cup_{l\geq0} \sD^l$, contributes $b
^{-|X|}\delta_{n(X)/b^{|X|}}$ where $|X|$ is the length of $X$ and $n(X)$ is
the integer obtained from $X=(d_l,\dots,d_1)$ as $d_l b^{l-1} +\dots + d_1$.
For the empty string, $n(X)$ is defined to be the integer $0$.  If $0\in A$
then many strings map to the same $b$-imal number $x(X)=n(X)/b^{|X|}$ as we
may always add trailing zeros, i.e.\@ make $X$ longer by shifting it to the
left, filling up with zeros on the right: this does not modify $x(X)\in[0,1)$.
The measure $\mu$ is first defined on $\bsD$ and then
pushed to $[0,1)$ via $X\mapsto x(X)$, the same notation is kept.

In the special case where $A=\emptyset$, all the mass comes from the
none-string (i.e.\@ the empty string) and $\mu=\delta_0$.  If $A=\{0\}$ then
$\mu=(1-1/b)^{-1}\delta_0$.

According to \cite[Section 3]{burnolkempner}, the measure $\mu$ has a finite
mass which is equal to $b/\#E$.  As is explained in \cite[Section
4]{burnolirwin} to which we refer for the elementary details, as the total
mass of the discrete measure $\mu$ is finite, we can use it to integrate
arbitrary bounded functions on the interval $[0,1)$, i.e.\@ they don't have to
be continuous.  And according to \cite[Lemma 7]{burnolkempner} any bounded (or
non-negative) function verifies the following integral identity:
\begin{equation}\label{eq:intg}
  \int_{[0,1)}g(x)\mu(\dx) = g(0) + \int_{[0,1)}\frac1b \sum_{a\in A}g(\frac{a+x}b)\mu(\dx)
\end{equation}
The integral equation \eqref{eq:K} representing $K(b,E)$ is immediately
checked from the given definition of $\mu$. And applying the above
\eqref{eq:intg} to \eqref{eq:K} expresses $K(b,E)$ as the sum of the $U(a)$'s
for the admissible non-zero digits $a\in A$ (as is evident a priori from
\eqref{eq:Unseries}).  Applying the integration lemma \eqref{eq:intg} to
$U(a)$ itself produces recursive identities which allow to replace the digits
$a\in A\setminus\{0\}$ with integers having at least two, or at least three,
or more, digits in base $b$.  This is the basis of the approach in
\cite{burnolkempner,burnolirwin} which leads to geometrically convergent
series whose coefficients involve the moments $u_m = \mu(x^m)$, or the
complementary moments $v_m=\mu((1-x)^m)$.  These moments are rational in the
base $b$ (and are also dependent on the excluded set $E$, of course).  They
obey linear recurrences of which we will give a few examples later.

\section{Digamma function and bounds for Kempner sums}

We combine the log-like formula \eqref{eq:K} with the integration
lemma \eqref{eq:intg} in another manner, using the digamma function $\psi(x)
=\frac{d}{dx}\log\Gamma(x)$.  We insert into \eqref{eq:K} its functional
equation \cite[1.7.1 (8)]{erdelyiI}:
\begin{equation}\label{eq:psi}
  \frac1x = \psi(x+1) - \psi(x)
\end{equation}
and then apply the integration lemma to the second part, because the function
$-\psi(x)$ still has the $1/x$ singularity at the origin and we want to get
rid of this to achieve a representation of $K(b,E)$ as an integral on the
full half-open interval $[0,1)$.  In applying the integration lemma we observe
that the sum over the digits corresponds to a decomposition of the original
integration range in successive intervals of length $b^{-1}$, but only those
intervals for which first fractional digit is admissible contribute.  When the
original function as here $\psi(x)\mathbf{1}_{[b^{-1},1)}(x)$ vanishes for
$x<b^{-1}$ there will be no contribution from the digit $a=0$ even if it is
admissible.

So $K(b,E) = I - J$ say, where $I=\int_{[b^{-1},1)}\psi(x+1)\mu(\dx)$ will be
treated later and
\begin{align*}
  J &= \int_{[b^{-1},1)} \psi(x)\mu(\dx)\\
&=\int_{[0,1)}\frac1b \sum_{a\in A,a>0} \psi(\frac{a+x}b)\mu(\dx)\\
&=\int_{[0,1)}\Bigl(\frac1b\sum_{1\leq a <b} \psi(\frac{a+x}b) -
\frac1b\sum_{a\in E\setminus\{0\}} \psi(\frac{a+x}b)\Bigr)\mu(\dx)
\end{align*}
Now enters the $b$-plication formula for the digamma
function \cite[1.7.1. (12)]{erdelyiI}:
\begin{equation*}
  \frac1b\sum_{0\leq a <b} \psi(\frac{a+x}b) = \psi(x) - \log(b)
\end{equation*}
We need a variant of the above where the singularity at $x=0$ is removed, and this is:
\begin{equation}\label{eq:psiadd}
  \frac1b\sum_{1\leq a \leq b} \psi(\frac{a+x}b) = \psi(x+1) - \log(b)
\end{equation}
Indeed it follows from the functional equation \eqref{eq:psi} that
\begin{equation*}
    \psi(x+1) -\frac1b\psi(\frac xb+1) = \psi(x)-\frac1b\psi(\frac xb)
\end{equation*}
With \eqref{eq:psiadd} we now get
\begin{equation}\label{eq:J}
  J =  -\log(b)\underbrace{\mu([0,1))}_{=b/\#E} + \int_{[0,1)} \psi(x+1)\mu(\dx)
  \begin{aligned}[t]
   & - \frac1b\int_{[0,1)}\psi(\frac xb+1)\mu(\dx)\\
   & - \int_{[0,1)}\frac1b\sum_{a\in E\setminus\{0\}} \psi(\frac{a+x}b)\mu(\dx)
  \end{aligned}
\end{equation}
We had left aside
\begin{equation*}
  I = \int_{[b^{-1},1)} \psi(x+1)\mu(\dx)
  = \int_{[0,1)} \psi(x+1)\mu(\dx)-\int_{[0,b^{-1})} \psi(x+1)\mu(\dx)
\end{equation*}
We handle the last integral via the integration lemma \eqref{eq:intg} for the
function $g(x)= \psi(x+1)$ for $0\leq x < b^{-1}$ and zero elsewhere.  We must
be careful to distinguish whether $0\in A$ or not:
\begin{equation*}
  \int_{[0,b^{-1})} \psi(x+1)\mu(\dx)=\psi(1)+
  \begin{cases}
    0&(0\notin A)\\
    \frac1b\int_{[0,1)}\psi(\frac xb + 1)\mu(\dx)&(0\in A)
  \end{cases}
\end{equation*}
So we obtain
\begin{equation}
  \label{eq:I}
  I = \int_{[0,1)} \psi(x+1)\mu(\dx)-\psi(1) 
     - \delta_{0\in A}\frac1b\int_{[0,1)} \psi(\frac xb +1)\mu(\dx)
\end{equation}
Combining \eqref{eq:I} and \eqref{eq:J} gives:
\begin{theorem}\label{thm:main}
  Let $E_1=E$ if $0\notin E$ and $E_1=(E\cup\{b\})\setminus\{0\}$ if
  $0\in E$.  The Kempner sum for base
  $b$ and excluded set of digits $E$ has value:
  \begin{equation*}
    K = \frac{b}{\#E}\log(b) 
        + \frac1b\int_{[0,1)}\sum_{a\in E_1} \Bigl(\psi(\frac{a+x}b)-\psi(1)\Bigr)\mu_E(\dx)
  \end{equation*}
  where $\psi$ is the digamma function.  We have indexed $\mu$ by $E$ to
  stress that it depends upon it.
\end{theorem}
\begin{proof}
  If $0\in E$ then it does not belong to $A$ and the term
  $\frac1b\int_{[0,1)} \psi(\frac xb +1)\mu_E(\dx)$ from \eqref{eq:J} has no
  compensation in equation \eqref{eq:I}. It gives an additional contribution
  with integrand $\psi((a+x)/b)$ using $a=b$ in the above equation.
  We can then move $\psi(1)$ inside the integral:  in all cases there
  are $\#E$ contributions and $\mu_E([0,1))=b/\#E$.
\end{proof}
As the $\psi$ function is strictly increasing on $(0,\infty)$ we can state
some immediate corollaries of Theorem \ref{thm:main}.  Recall that $\mu_E([0,1))
= b/\#E$, so each $a\in E_1$ contributes a quantity $s_{b,E}(a)$ which verifies
\begin{equation}\label{eq:boundsperdigit}
  \frac{\psi(\frac ab)- \psi(1)}{\#E} \leq s_{b,E}(a) <  
  \frac{\psi(\frac {a+1}b)- \psi(1)}{\#E}
\end{equation}
The lower bound will be an equality if and only if $\mu_E$ is concentrated on
$\{0\}$ and then the lower bound is an equality for all $a\in E_1$.  This
happens only if $E$ contains all positive digits, in which case $K(b,E)$ is an
empty sum and has value zero.  Let's check all is fine.  If $E=\sD$, then
$\mu_E = \delta_0$ where $\delta_0$ is the Dirac point mass at the origin: the
sole admissible string is the empty string.  We need to check
\begin{equation*}
  0 = \log(b) + b^{-1}\sum_{1\leq a\leq b} \bigl(\psi(\frac ab)- \psi(1)\bigr)
\end{equation*}
which is indeed true from \eqref{eq:psiadd} applied with $x=0$. If
now $E=\sD\setminus\{0\}$ then $\mu_E=(1 - b^{-1})^{-1}\delta_0$ (from all
strings containing only zero digits) and we need to check
\begin{equation*}
  0 = \frac{b \log(b)}{b-1} 
     + (b-1)^{-1}\sum_{1\leq a<b} \bigl(\psi(\frac ab)- \psi(1)\bigr)
\end{equation*}
and again it works.

From \eqref{eq:boundsperdigit} we have $s_{b,E}(a)<0$ except if $a=b$ so a
trivial corollary is that $K(b,E)<(\#E)^{-1}b\log(b)$ if $0$ is not an
excluded digit.  Somewhat more is true:
\begin{proposition}\label{prop:blogbsurm}
  One has $K(b,E)< \frac{b}{\#E}\log(b)$ in all cases except for the sole
  exception which is $E=\{0\}$, then $K(b,\{0\})>b\log(b)$.
\end{proposition}
\begin{proof}
  We have already seen that for the inequality to be false, $0$ must be among
  the excluded digits.  And from the previous discussion if $E=\{0\}$ then
  $K(b,E)=b\log(b) + s_{b,E}(0)>b\log(b)$.  Assume now that $E$ contains $0$
  and some positive digit $a$, and is not all of $\sD$ (as $K(b,\sD)=0$ the
  case $E=\sD$ is ok).  So $A=\sD\setminus E$ contains some positive digit or
  digits, and we can refer to theorems of \cite{burnolkempner} as this
  reference assumes $A$ contains at least one positive digit.  We will show
  that $\psi(\frac{a}b+\frac xb)+\psi(1+\frac xb)-2\psi(1)$ gives a negative
  result when integrated against $\mu_E$.  We use the strict concavity to
  obtain $\psi(\frac{a+x}b)+\psi(1+\frac
  xb)-2\psi(1)<\zeta(2)(\frac{a+x}{b}-1)+\zeta(2)\frac xb$ for $0\leq x \leq
  1$ and now inject $a\leq b-1$ so the integrand is bounded above (strictly)
  pointwise by $\zeta(2)b^{-1}(2x-1)$.  So the question is now whether
  $2u_1<u_0$ with $u_m=\mu_E(x^m)$.  In \cite[Prop. 8]{burnolkempner} the
  formulas $u_0 = \frac{b}{b-N}$ with $N=\#A=b-\#E\leq b-2$ and $u_1
  =(b^2-N)^{-1}u_0\sum_{d\in A} d$ are given and we thus have to check if
  $2\sum_{d\in A} d < b^2 - N$.  But $2\sum_{d\in A} d< b(b-1)=b^2-b$ so this
  is true with some margin.
\end{proof}
\begin{proposition}\label{prop:monotonie}
  Let $E$ be of cardinality $p<b$.  Then $K(b,E)\leq
  K(b,\{0,b-1,\dots,b-p+1\})$ with equality only for $E=\{0,b-1,\dots,b-p+1\}$.
\end{proposition}
\begin{proof}
  Consider an excluded \emph{positive} digit $a\in E$ such that $a<b-1$ and
  $a+1\notin E$, if there is one.  Let $E'=(E\setminus\{a\})\cup\{a+1\}$.  We
  claim that $K(b,E)<K(b,E')$.  For this let $\cA$ be the set of positive
  integers with no $b$-ary digits in $E$ and $\cA'$ be the analogous set of
  positive integers with no $b$-ary digits in $E'$.  Let $n\in \cA$ and
  consider the digits $d_\ell,\dots, d_1$ of its minimal representation
  (i.e. $d_\ell>0$), $n=d_\ell b^{\ell-1}+\dots + d_1 b^0$. None of the digits
  is equal to $a$, but maybe some are equal to $a+1$.  Define $f(n)$ to be
  that integer obtained from replacing all digits equal to $a+1$ by $a$.  As
  $a>0$ notice that the minimal representation of $f(n)$ via $b$-digits is the
  one we just described: the leading digit did not become zero.  Certainly
  $f(n)\in\cA'$.  And certainly $n\geq f(n)$ with equality only if no digit
  was equal to $a+1$.  Now, conversely, suppose a positive integer $m$ with
  $\ell$ digits is given in $\cA'$.  Define $g(m)$ via the replacement of all
  digits equal to $a$ by $a+1$.  As $a<b-1$, this replacement gives us a legit
  $b$-ary representation (i.e. without carry operations).  We have $g(f(n))=n$
  and $f(g(m))=m$.  So $f$ establishes a bijection.  The Kempner series (they
  are not empty here...) have positive terms, so any rearrangement gives the
  same sum.  As for example $f(a+1)=a<(a+1)$ and
  generally $f(n)\leq n$, we deduce that $K(b,E)<K(b,E')$.

  Proceeding iteratively (moving the largest digit of $E$ to $b-1$, then the
  next largest to $b-2$, etc\dots) we obtain $K(b,E)\leq K(b,E'')$ where
  either $E''= T = \{0,b-1,\dots,b-p+1\}$ or $E''=U = \{b-1,\dots,b-p\}$
  depending on whether $0\in E$ or not, and with equality only if $E=E''$.  We
  now prove $K(b,U)<K(b,T)$ using Theorem \ref{thm:main}.  We have to be
  careful of course that the measures are not the same.  But we can use the
  bounds \eqref{eq:boundsperdigit}.  On one hand we obtain
  \begin{equation*}
    K(b,T) \geq \frac{b}p \log(b) + 
      \frac 1p \sum_{a=b-p+1}^{a=b} \bigl(\psi(\frac{a}{b})- \psi(1)\bigr)
  \end{equation*}
  and on the other hand
  \begin{equation*}
    K(b,U) < \frac{b}p \log(b) + 
      \frac 1p \sum_{a=b-p}^{a=b-1} \bigl(\psi(\frac{a+1}{b})- \psi(1)\bigr)
  \end{equation*}
  The conclusion follows.
\end{proof}
\begin{remark}
  From the \textsf{Maple\texttrademark} code provided with
  \cite{burnolkempner} we know that $K(10,\{0,9\})$ is approximately
  $\np{11.490785103824471}\approx \np{0.499}\times10\log(10)$.  And we can
  confrim that this is the maximal value for $b=10$ and two excluded digits.
  The maximal value for $b=10$ and three excluded digits is found to be
  $K(10,\{0,8,9\})\approx\np{7.543171528424965}\approx\np{0.3276}\times10\log(10)$.
  The maximal value for $b=10$ and four excluded digits is
  $K(10,\{0,7,8,9\})\approx\np{5.501015712594091}\approx\np{0.2389}\times
  10\log(10)$.  These numerical results%
  \footnote{Of course they were not obtained from Theorem \ref{thm:main} but
    using the algorithm and code from \cite{burnolkempner}; readers who have access to
    \textsf{Mathematica\texttrademark} can use the code of Robert Baillie at
    \cite{baillie2008}, the pdf file gives on page 16 the example
    $K(10,\{8,9\})\approx\np{11.2915816168}$ which is also what the algorithm
    of \cite{burnolkempner} produces; in \cite{schmelzerbaillie} the example
    is given of $K(10,\{0,2,4,6,8\})$ with 100 fractional digits and it
    matches exactly what the \textsf{Maple\texttrademark} code of
    \cite{burnolkempner}, which is based on a completely different algorithm,
    produces in that case.  And the same perfect match is observed with the
    value of $K(10,\{1,3,5,7,9\})$ with 100 fractional decimal digits.}
  appear to match well the predictions from Propositions \ref{prop:blogbsurm}
  and \ref{prop:monotonie}.%
%

We went through extensive numerical verifications with larger bases (as
admittedly $b=10$ is not close to infinity...).  We used for this
the \cite{burnolkempner} algorithm, which originated in the same general framework
of the measures and their moments, which has led us to Theorem \ref{thm:main},
but readers are invited to check with the code of \cite{baillie2008}.
\end{remark}
\enlargethispage{\baselineskip}
We now turn to some consequences of equation \eqref{eq:boundsperdigit} when
$E$ is a singleton.
\begin{proposition}\label{prop:Kbdbounds}
  There holds, for $d=0$:
  \begin{equation*}
    b\log(b)<K(b,\{0\})<b\log(b)+\psi(1+b^{-1})-\psi(1)<b\log(b)+\zeta(2) b^{-1}
  \end{equation*}
  and, for $b>2$ and $0<d<b$:
  \begin{equation*}
    b\log(b) + \psi(\frac db)-\psi(1) < 
    K(b,\{d\}) <
    b\log(b) + \psi(\frac {d+1}b)-\psi(1)
  \end{equation*}
  Hence $K(b,\{d\})$ is strictly increasing for positive increasing $d$, and
  bounded above strictly by $b\log(b)$ for positive $d$'s.
\end{proposition}
\begin{proof}
  Both statements follow directly from \eqref{eq:boundsperdigit} and from our
  discussion of when the lower bound there is not strict: for $E$ a singleton
  this can happen only for $b=2$ and $d=1$, which gives $K(2,\{1\})=0$ and is
  not among the considered cases in the proposition.
\end{proof}
\begin{remark}
  The bounds for
  $K(b,\{0\})$ are a quantitatively precise version of Fine's $b\log(b)+
  O(b^{-1})$ result from \cite{segalleppfine1970}.

  Note that $-d^{-1}b<\psi(db^{-1})-\psi(1)$ so the estimate
  $b\log(b)+O(b)$ of Kl{\o}ve \cite{klove1971} is confirmed in that case of a single
  excluded digit.

  For $b=2$ and $d=1$ the measure $\mu_E$ is $2\delta_0(x)$ and the formula of
  Theorem \ref{thm:main} gives $K(2,1) = 2\log(2) - \psi(1)+\psi(\frac12)=0$,
  which is the correct value.

  For $b=2$ and $d=0$, $K(2,0)$ is the Erdös-Borwein constant
  $\sum_{n>0}(2^n-1)^{-1}$ of numerical value
  $\approx{\np{1.60669515241529}}$.  The upper bound from Proposition
  \ref{prop:Kbdbounds} is $2\log(2)+\psi(\frac32)-\psi(1) = 2$.
\end{remark}

\section{Zeta values and asymptotic expansions}

\begin{proposition}
  Let $F=\{b-a, a\in E_1\}$, so $F\subset\sD$ ($F$ is the set of residue
  classes modulo $b$ of the $b-a$'s for $a\in E$).  The Kempner sum for base
  $b$ and excluded set of digits $E$ has value:
  \begin{equation*}
    K = \frac{b}{\#E}\log(b)
      - \int_{[0,1)}\frac1b\sum_{d\in F} \bigl(\psi(1)-\psi(1 - \frac{d-x}b)\bigr)\mu_E(\dx)
  \end{equation*}
\end{proposition}
\begin{proof}
  This is only a matter of replacing $a+x$ by $b - (b-a -x)$ and setting
  $d=b-a$ for $a\in E_1$, in the formula of Theorem \ref{thm:main}.
\end{proof}
Using the power series expansion \cite[1.17 (5)]{erdelyiI} 
$
  \psi(1)-\psi(u) = \sum_{m=1}^\infty \zeta(m+1)(1-u)^{m}
$
which is valid for $|u-1|<1$ with normal convergence if $|u-1|\leq 1-\eta$, $0<\eta<1$,
we obtain:
\begin{theorem}\label{thm:series}
  The Kempner sum for base $b$ and excluded set of digits $E$ has value:
  \begin{equation*}
    K(b,E) = \frac{b}{\#E}\log(b) 
              -\sum_{m=1}^\infty\frac{\zeta(m+1)\sum_{d\in F}v_m(d)}{b^{m+1}}
  \end{equation*}
  where $F$ is the set $b-E$ modulo $b$ and 
  \begin{equation*}
    v_m(d) = \int_{[0,1)}(d-x)^m\mu_E(\dx)
  \end{equation*}
\end{theorem}
It is amusing to check that the formula does work also when $A$ is empty or
the singleton $\{0\}$.
Then $\mu_E$ reduces to $c \delta_0$ with $c=b/\#E$, hence
$v_m(d)=c d^m$,
\begin{equation*}
  \sum_{m=1}^\infty\frac{\zeta(m+1)d^m}{b^{m+1}} =
  \sum_{n=1}^\infty\sum_{m=1}^\infty\frac{d^m}{(nb)^{m+1}}= \sum_{n=1}^\infty\Bigl(\frac{1}{nb-d}-\frac{1}{nb}\Bigr)
\end{equation*}
and after summing from $d=0$ (which can always be added) to $b-1$ we end up with $\lim\;
(H_{nb}-H_n)=\log(b)$, hence the formula correctly predicts $K(b,E)=0$.

One always has for $d\geq1$ the lower bound $v_m(d)\geq d^m \mu_E(\{0\})$ with
$\mu_E(\{0\})$ being $1$ or $b/(b-1)$ depending on whether $0\in E$ or not.  So
with $d_1 = \max F$ and except if $E=\{0\}$, the series (even putting aside
the problem of needing the values $\zeta(n)$, $n\geq2$) converges no better
than a sum of powers of $d_1/b$ (for example if $E=\{d\}$ is a singleton with
$d>0$ one has $d_1=b-d$ and if $d$ is fixed and $b$ large we obtain a poorly
convergent series).  In contrast, the ``level $2$'' series of
\cite{burnolkempner} converge as fast as sums of powers of $\pm1/b$ and ``level
$3$'' series have an efficiency comparable to a geometric series with ratio
$\pm1/b^2$.  Further manipulations would be needed here to obtain such
efficiency.

The two most favourable cases (if using Theorem \ref{thm:series} with no further
elaborations) are $E=\{0\}$ or $E=\{b-1\}$, which we discuss briefly now.

If $E=\{0\}$, then $F=\{0\}$ and $v_m(0)=(-1)^mu_m$, $u_m=\mu_E(x^m)$, and:
\begin{equation}\label{eq:Kb0}
  K(b,\{0\}) = b\log(b) + \sum_{m=1}^\infty (-1)^{m-1}\frac{\zeta(m+1)u_m}{b^{m+1}}
\end{equation}
The sequence $(u_m)$ is positive and strictly decreasing, so we get an
alternating series with decreasing absolute values.  It is also known from
\cite[Prop. 10]{burnolkempner} that $1 < (m+1)u_m < b$.  So the speed of
convergence is roughly like the one of an alternating geometric series of
ratio $1/b$.  And furthermore we get ``explicit'' lower and upper bounds such as:
\begin{equation}\label{eq:Kb0bounds}
  b\log(b) + \frac{\zeta(2) u_1}{b^2} - \frac{\zeta(3)u_2}{b^3} < K(b,0) < 
b\log(b) + \frac{\zeta(2) u_1}{b^2} - \frac{\zeta(3)u_2}{b^3} + \frac{\zeta(4)u_3}{b^4}
\end{equation}
According to \cite[Prop. 8]{burnolkempner}, the recurrence relation of
the $u_m$'s in this case  $E=\{0\}$ is:
\begin{equation}
  (b^{m+1} - b + 1) u_m = \sum_{j=1}^m \binom{m}{j}(\sum_{1\leq a <b}a^j )u_{m-j}
\end{equation}
One deduces from it (using $u_0=b$):
\begin{align}
  \frac{u_1}{b^2} &=\frac1{2b} - \frac1{2b(b^2-b+1)}\\
  \frac{u_2}{b^3} &= \frac1{3b^2} -\frac{5b^2-5b+2}{6b^2(b^2-b+1)(b^3-b+1)}
\end{align}
Using now the bound $0<u_3<b/4$ we get a concrete estimate of
$K(b,\{0\})$ with an error bounded by $\zeta(4)/(4b^3)$.  And as from the above
clearly $u_1/b^2 = 1/(2b) + O(1/b^3)$ and $u_2/b^3 = 1/(3b^2) + O(1/b^3)$ we
obtain
\begin{equation}
  K(b,\{0\}) = b\log(b)+\frac{\zeta(2)}{2b}-\frac{\zeta(3)}{3b^2} + O(b^{-3})
\end{equation}
More generally, keeping terms up to $m=M$ in the series \eqref{eq:Kb0} the
error is bounded by the first ignored term so is $O(b^{-M-1})$ as $0\leq
u_{M+1} \leq b$.  Each contribution of the partial sum is a rational function
of $c=b^{-1}$ and it is known from \cite{burnolkempner} that $c u_m$ is
regular at $c=0$ (with value $1/(m+1)$, so the $m$\textsuperscript{th} term is
$\sim (-1)^{m-1}\zeta(m+1) c^m/(m+1)$.  Hence the existence of a full
asymptotic expansion in descending powers of $b$ to all orders.

With a bit more effort we could derive \eqref{eq:2} inclusive of its $b^{-3}$
term, but we refer the reader to \cite{burnolasymptotic} for the formulas in
terms of zeta values for the first five coefficients.

Apart from $E=\{0\}$ the most favourable situation is when $E=\{b-1\}$ (which
was the case with $b=10$ in Fischer work \cite{fischer}).  Then $F = \{1\}$
and the  series from Theorem \ref{thm:series} is
\begin{equation}\label{eq:Kbmoinsun}
  K(b,\{b-1\}) = b\log(b) - \sum_{m=1}^\infty \frac{\zeta(m+1)v_m(1)}{b^{m+1}}
\end{equation}
The coefficients $v_m(1)=\int_{[0,1)}(1-x)^m\mu_E(\dx)$ are bounded above by
$b$ and below (from the mass at the origin) by $b/(b-1)$.  The series has
about the same speed of convergence as a geometric series with ratio $1/b$.
This sequence $(v_m(1))$ is a particular case of ``complementary moments'' as
considered in \cite[\S7]{burnolirwin}, let's shorten the notation to
$(v_m)$. The recurrence in this case with a single excluded digit $d=b-1$ (so
$d'=0$ in the notation of \cite[\S7]{burnolirwin}) is:
\begin{equation}
  (b^{m+1} - b + 1) v_m = b^{m+1} + \sum_{j=1}^m \binom{m}{j}(\sum_{1\leq a <b}a^j )v_{m-j}
\end{equation}
and one obtains in particular:
\begin{align}
  \frac{v_1}{b^2} &=\frac1{2b} + \frac {2b-1}{2b ( {b}^{2}-b+1 ) }\\
  \frac{v_2}{b^3} &= \frac1{3b^2} + \frac
    {6{b}^{4}-5{b}^{2}+5b-2}{6{b}^{2} ( {b}^{2}-b+1 )( {b}^{3}-b+1) }
\end{align}
The contribution to \eqref{eq:Kbmoinsun} of the terms with $m\geq3$ is
$O(b^{-3})$.  Using the above we have $v_1/b^2 = 1/(2b) + 1/b^2 + O(1/b^3)$
and $v_2/b^3 = 1/(3b^2) + O(1/b^3)$ so we obtain
\begin{equation}
  K(b,\{b-1\}) = b\log(b) - \frac{\zeta(2)}{2b} - \frac{3\zeta(2)+\zeta(3)}{3b^2}
  + O(b^{-3})
\end{equation}
With some more effort we could obtain the third term of \eqref{eq:3} as stated
in the introduction, thus reproducing our earlier result from
\cite{burnollargeb}.  Using $0\leq v_m\leq v_0=b$ we obtain that for any $M$,
keeping the contributions from $m=1$ to $m=M$, the error in the approximation
to $K(b,\{b-1\})$ is $O(b^{-M-1})$.  And each such contribution is a rational
function in $b$, which, as a function of $c=b^{-1}$ is equivalent to
$\zeta(m+1)c^m/(m+1)$ for $c\to0$.  Hence the existence of an asymptotic
expansion to all orders in descending powers of $b$, a result which improves
upon \cite{burnollargeb}.

A more thorough expansion of these ideas, for fixed $d\in\sD=\{0,\dots,b-1\}$,
is the topic of \cite{burnolasymptotic}, which uses Theorem \ref{thm:main} as
starting point rather than Theorem \ref{thm:series}.

\section{Two types of linear recurrences for \texorpdfstring{$(v_m(d))$}{(vm(d))}}

We examine the quantities $v_m(d) = \int_{[0,1)} (d-x)^m\mu_E(\dx)$ for $d\in F$.  We can
of course expand in powers of $x$ or $1-x$ hence obtain expressions as linear
combinations of the $u_m=\int_{[0,1)} x^m\mu_E(\dx)$ and complementary moments
$v_m=\int_{[0,1)} (1-x)^m\mu_E(\dx)$ as considered in \cite{burnolkempner} and
\cite{burnolirwin}.  Or we can apply the integration lemma \ref{eq:intg} which gives
\begin{align}
  v_m(d) &= d^m + \frac1b \sum_{a\in A} \int_{[0,1)} (d - \frac{a+x}{b})^m\mu_E(\dx)\notag\\
&= d^m + b^{-m-1}\int_{[0,1)}\sum_{a\in A}(bd-a-d+d-x)^m\mu_E(dx)\notag\\
&=d^m+b^{-m-1}\sum_{j=0}^m \binom{m}{j}\bigl(\sum_{a\in A}(bd-a-d)^j\bigr) v_{m-j}(d)\notag\\
(b^{m+1} - \#A)v_m(d) &= b^{m+1}d^{m} + \sum_{j=1}^m \binom{m}{j}\bigl(\sum_{a\in A}(bd-a-d)^j\bigr) v_{m-j}(d)\label{eq:recurr1}
\end{align}
which expresses $v_m(d)$ in terms of the $v_{m-j}(d)$ with
$1\leq j\leq m$.  In passing the recursion confirms $v_0(d) = b/(b-\# A) =
b/\#E$.

We can also regroup the $v_m(d)$'s in a generating function
$F(t) = \sum \frac{v_m(d)}{m!} t^m$:
\begin{equation*}
  F(t) = \int_{[0,1)} e^{t(d-x)}\mu_E(\dx)
\end{equation*}
If we apply the integration lemma \eqref{eq:intg} we end up with a functional
equation which is equivalent to \eqref{eq:recurr1}.
\begin{align*}
  F(t) &= e^{td}+\frac1b\sum_{a\in A}\int_{[0,1)} e^{t(d-\frac ab - \frac xb)}\mu_E(\dx)
\notag\\
       &= e^{td} + \frac1b \sum_{a\in A}e^{t(d-\frac ab)}e^{-t\frac db}F(\frac tb)
\notag\\
&=e^{td} + \frac1b \sum_{a\in A}e^{t\frac{bd-a-d}b}F(\frac tb)
\end{align*}
Originally Kempner sums were considered with $E$ a singleton and we would
like to have a summation over $E$ rather than $A$, hence:
\begin{align}
F(t)&=e^{td}+  \frac1b \sum_{0\leq i<b}e^{t\frac{bd-i-d}b}F(\frac tb)
                 - \frac1b \sum_{a\in E}e^{t\frac{bd-a-d}b}F(\frac tb)
\notag\\
(e ^{t/b}-1)F(t)&= e^{td}(e^{t/b}-1)
\begin{aligned}[t]
  &+ \left(e^{t\frac{bd+1-d}b} - e^{t\frac{(b-1)(d-1)}b}\right.\\
  &- \left.\sum_{a\in E}e^{t\frac{bd-a-d}b}(e^{t/b}-1)\right)\frac1b F(\frac tb)
\end{aligned}
\label{eq:recurr2gf}
\end{align}
The situation is very favourable when $E=\{b-1\}$.  Then $d=1$, and there is
only a single $a\in E$ which is $b-d$ so $bd-a-d=0$.  Let's work out this
case:
\begin{equation*}
  (e ^{t/b}-1)F(t)= e^{t}(e^{t/b}-1)
+ \Bigl(e^{t} - 1 - (e^{t/b}-1)\Bigr)\frac1b F(\frac tb)
\end{equation*}
Expanding in powers of $t$ gives now the following relations among the
$c_m=v_m(1)$ for this special case:
\begin{gather}
  \sum_{j=1}^{m}\binom{m}{j}b^{-j}c_{m-j}=
(1+\frac1b)^{m}-1 + \sum_{j=1}^m\binom{m}{j}(1  - b^{-j})b^{-(m-j+1)}c_{m-j}
\notag\\
 \sum_{j=1}^{m}\binom{m}{j}(b^{m+1-j}-b^j + 1)c_{m-j} 
= b((b+1)^m - b^m)
\label{eq:recurrspecial}
\end{gather}
The general case \eqref{eq:recurr2gf} gives the following relations among the
$v_m(d)$, for each $d\in F$:
\begin{equation*}
b((db+1)^m - (db)^m)
=  \sum_{j=1}^{m}\binom{m}{j}\Bigl(
\begin{aligned}[t]
  &b^{m+1-j}\\
  &-(bd+1-d)^j+(b-1)^j(d-1)^j\\
  &+\sum_{a\in E} \bigl((bd-a-d+1)^j-(bd-a-d)^j\bigr)
  \Bigr)v_{m-j}(d)
\end{aligned}
\end{equation*}
Here is thus the general explicit recurrence complementary to \eqref{eq:recurr1}.
\begin{align}\label{eq:recurr2}
  (m+1)&(b^{m+1}-\#A) v_m(d) = 
   b((db+1)^{m+1} - (db)^{m+1}) \\\notag
&-\sum_{j=1}^m\binom{m+1}{j+1}\Bigl(
\begin{aligned}[t]
  &b^{m+1-j}\\
  &-(bd+1-d)^{j+1}+(b-1)^{j+1}(d-1)^{j+1}\\
  &+\sum_{a\in E} \bigl((bd-a-d+1)^{j+1}-(bd-a-d)^{j+1}\bigr)\\
  &\Bigr)v_{m-j}(d)
\end{aligned}
\end{align}
We never used $d\in F$ so \eqref{eq:recurr2} (and \eqref{eq:recurr1}) hold for all
$d$'s even non-integers.

\section{Concluding remarks}

Consider the case $b=10$, $A=\sD\setminus\{9\}$.  Then $K=K(10,\{9\})$ is the
original Kempner sum \cite{kempner} and the Theorem \ref{thm:series} says
\begin{equation*}
K = 10 \log(10) - \sum_{m=1}^\infty \frac{\zeta(m+1)v_m}{10^{m+1}}  
\end{equation*}
where the sequence $(v_m)$ is defined as $v_m =
\int_{[0,1)}(1-x)^m\mu_E(\dx)$.  This looks like the ``interessante Resultat
1'' \cite[eq. (8)]{fischer} obtained by Fischer in 1993, and will prove being
exactly it once one knows that the $v_m$'s here, i.e.\@ the $c_m$'s verifying
\eqref{eq:recurrspecial} from the previous section, are the same as the
$\beta_m$'s from \cite{fischer}.  And indeed the special recurrence
\eqref{eq:recurrspecial} is exactly \cite[eq. (7)]{fischer} (the initial term
$\beta_0$ is determined by the formula itself).

Besides, one convinces oneself indeed easily that the Fischer linear form
``$l(f)=(1-A)^{-1}(f)(0)$'' on $C([0,1])$ from \cite[Sect. 2]{fischer} is
nothing else than integration on $[0,1)$ against the measure $\mu_E$ from
\cite{burnolkempner}, specialized to $b=10$ and $E=\{9\}$.

\providecommand\bibcommenthead{}
\def\blocation#1{\unskip}

\end{document}